\documentclass{amsart}
\usepackage{amssymb, amsmath, latexsym, a4}
\usepackage[latin1]{inputenc}
\usepackage[all]{xy}

\def\xto#1{\xrightarrow[]{#1}}

\def\I{\mathbb{I}}
\def\q{{\frak q}}

\def\p{{\frak p}}

\def \hom{\mathop{\sf Hom}\nolimits}

\def\1{^{-1}}
\newtheorem{De}{Definition}[section]
\newtheorem{Th}[De]{Theorem}
\newtheorem{Pro}[De]{Proposition}
\newtheorem{Le}[De]{Lemma}
\newtheorem{Co}[De]{Corollary}
\newtheorem{Rem}[De]{Remark}

\begin{document}

\title{On the spectrum of  monoids and semilattices}

\author[I. Pirashvili]{Ilia  Pirashvili}
\email{ilia\_p@ymail.com}

\maketitle

\section{Introduction}
In the last years there was considerable interest in the spectrum of commutative monoids \cite{kato},\cite{b},\cite{deitmar1},\cite{deitmar2},\cite{cc1},\cite{cc2}, \cite{chu}, \cite{cort}. These objects play the same role in the theory of schemes over the 'field with one element' as the spectrum of rings played in the theory of schemes over rings. 
See also \cite{lopez} for more about geometry over the field with one element.

The aim of this short note is to prove a useful lemma on the spectrum of commutative monoids and some of its immediate consequences. Our lemma claims that for any commutative monoid $M$ one has a natural isomorphism 
\begin{equation}\label{sph} Spec(M)\cong \hom(M, \I),\end{equation}
where $\hom$ is taken in the category of commutative monoids and 
$\I=Spec(\mathbb{N})$.
The monoid structure in the spectrum of a monoid, is given by the union of prime ideals. 
From isomorphism (\ref{sph}), one easily deduces the 'reduction isomorphism'
\begin{equation} Spec(M)\cong Spec(M^{sl}),\end{equation}
where $M^{sl}$ is $M$ modulo to the relation $a^2=a$. Thus the study of the spectrum of commutative monoids, reduces to the study of the spectrum of semilattices. Our second main result claims that if $L$ is a semilattice, then there is an injective map 
$$\alpha_L:L\to Spec(L),$$
which is bijective provided $L$ is  finitely generated (hence  finite). In particular these results give an effective way of computing $Spec(M)$ for an arbitrary finitely generated commutative monoid $M$.
We also show that $\alpha$ is ''natural'' in the following sense. If $f:L\to L'$ is a morphism of finitely generated  semilattices, then one has the following commutative diagram
 $$\xymatrix{L'\ar[r]^{f^\dagger}\ar[d]_{\alpha_{L'}}& L\ar[d]^{\alpha_L}\\
 Spec(L')\ar[r]_{f^{-1}}& Spec(L)}$$
 where  $f^\dagger$ denotes the right adjoint of $f$. The case of infinite semilatices is also considered. 
\\

Most of the results in this note were part of my 4th year masters thesis at UCL, under the supervision of Dr. Javier L\'{o}pez Pe\~na, who has introduced me to this subject. But Isomorphism (1), which is new, not only simplifies much, but also gives a new insight in the spectrum of monoids.

 \section{Prime ideals of commutative monoids} 
In what follows all monoids are commutative and they are written multiplicatively.
Let $M$ be a monoid, then a subset
$\frak{a}$ is called an ideal provided for any $a\in \frak{a}$ and
$x\in M$ one has $ax\in \frak{a}$. For an element $a\in M$, we let $(a)$ be the principal ideal $aM$. An ideal $\p$ is called \emph{prime}, provided $\p\not =M$ and the complement $\p^c$ of  $\p$ in $M$ is a submonoid. Thus an ideal 
$\frak{p}$ is prime iff $1\not \in \frak{p}$ and if $xy\in \frak{p}$
then either $x\in \frak{p}$ or $y\in \frak{p}$. We let $Spec(M)$ be the set of all prime ideals of $M$. It is equipped with a topology, where the sets
 $$D(a)=\{\p\in Spec(M)| a\not \in \p\}$$
form a bases of open sets. Here $a$ is an arbitrary element of $M$, see \cite{kato},\cite{deitmar1},\cite{deitmar2}. Observe that if 
$f:M_1\to M_2$ is a homomorphism of monoids then for any prime ideal
$\frak{q}\in Spec(M_2)$ the pre-image $f^{-1}(\frak{q})$ is a prime
ideal of $M_1$, hence any monoid homomorphism $f:M_1\to M_2$ gives rise to a continuous map 
 $$f^*:Spec(M_2)\to Spec(M_1); \ \ \ \q\mapsto f^{-1}(\q).$$
 It is obvious that the union of prime ideals is again a prime ideal. It is also clear that the empty set is a prime ideal, which is the least prime ideal, and the set of noninvertible elements of $M$ is a prime ideal, which is the greatest prime ideal. Thus $Spec(M)$ is a (topological) monoid with respect to union. 
 
Let $\I=\{0,1\}$ be the monoid with the obvious multiplication. We will equip the set $\I$ with the topology, where the subsets $\emptyset, \{1\}, \I$ are all open sets. Let $\mathbf{N}=\{1,t,t^2,\cdots\}$ be the free monoid with one generator (of course it is isomorphic to the additive monoid of natural numbers). Then $Spec(\mathbf{N})= \{\emptyset, (t)\}$. Observe that  $\emptyset$ is an open subset of  $Spec(\mathbf{N})$. Hence $\I\cong Spec(\mathbf{N})$ as topogical monoids.
\begin{Le}\label{spec=hom} For any commutative monoid $M$ one has a natural isomorphism of topological monoids
 $$Spec(M)\cong \hom(M,\I),$$
where the topology on the right hand side is induced by the the product topology on $\prod_{m\in M}\I$.
\end{Le} 

\begin{proof} Observe that $\{0\}$ is a prime ideal of $\I$ (in fact $Spec(\I)=\{\emptyset, \{0\}\} \cong \I$). Hence for any homomorphism $f:M\to \I$ we have $f^{-1}(0)\in Spec(M)$. In this way we obtain a map $\theta:\hom(M,\I)\to Spec(M)$. The inverse of this map is given as follows, for a prime ideal $\p\in Spec(M)$, we let $f_\p:M\to \I$ be the map given by $$f_\p(x)=\begin{cases} 0,\ {\rm if } \ x\in \p \\ 1, \ {\rm if } \ x\notin \p\end{cases}$$
Now it is easy to check that $f_\p:M\to \I$ is a homomorphism. Hence $\p\mapsto f_\p$ defines a map $Spec(M)\to \hom(M,\I)$ which is clearly a homomorphism and is the inverse of $\theta$. It is also trivial to check that  $\theta$ and its inverse are continuous.
\end{proof}
Since $\hom(-,\I)$ takes colimits to limits we have (compare pp. 5 and 6 in \cite{cort}).

\begin{Co}\label{zg} If $J$ is a poset and  $(M_j)_{j\in J}$ is a direct system of monoids indexed by $J$, then the natural map
$$Spec(colim_{j\in J}\, M_j)\to lim_{j\in J}\, Spec(M_J)$$
is an isomorphism.
\end{Co}

Recall that the category of monoids satisfying the identity $a^2=a$ is equivalent to the category of \emph{join semilattices} \cite{grillet}. By a join semilattice we mean  a poset $L$ with a least element such
that for any two elements $a,b\in L$ there exist the \emph{join} $a\vee b$, which is the least element among the elements $x$ such that $a\leq x$ and $b\leq
x$.  If  $M$ is a monoid such that $m^2=m$ holds for all $m\in M$, then one defines $x\leq y$ if  $xy=y$. It is easily seen that in this way we obtain a join semilattice. Conversely, if  $L$ is a join semilattice we  can consider $L$ as a monoid, with operation $$xy:=x\vee y.$$
The least element of $L$ is the unit in this monoid. It is clear that we have $x^2=x$ for all $x\in L$. 

We let ${\bf SL}$ be  the full subcategory of  monoids $M$ satisfying the identity  $x^2=x$ for  all $x$. For example we have $\I  \in {\bf SL}$. It is well-known (and trivial) that the  inclusion ${\bf SL}\subset {\bf Mon}$ has a left adjoint functor $M\mapsto M^{sl}$, where $M^{sl}$ is the quotient of $M$ by the smallest congruence $\sim$ for which $x\sim x^2$ for any $x\in M$. For example $\I=(\mathbb{N})^{sl}$. For the reader familiar to the tensor product of commutative monoids we also mention the isomorphism $M^{sl}\cong M\otimes \I$. However we will not make use of this fact. As it follows from Lemma \ref{spec=hom} and Corollary \ref{spsp} below, the object $\I$ is a dualizing object in the category of finite monoids satisfying  the identity $m^2=m$, that is for any such $M$ the following map is an isomorphism
$$ev: M\cong \hom(\hom(M,\I),\I),$$
where $(ev(m))(f)=f(m), m\in M, f\in \hom(M,\I).$

\begin{Le}\label{red} ({\bf Reduction lemma}).  For any monoid $M$, the canonical homomorphism $q:M\to M^{sl}$ yields the isomorphism
$$q^{-1}:Spec(M^{sl})\to Spec (M).$$
\end{Le}
\begin{proof} By definition of $M^{sl}$ for any $X\in {\bf SL}$ one has $\hom(M,X)=\hom(M^{sl}, X)$. By putting  $X=\I$ the result follows from Lemma \ref{spec=hom}.
\end{proof} 

As a consequence we obtain the following fact which sharpens Lemma 4.2 in \cite{deitmar2}. We will make use of the following description of $M^{sl}$. Thanks to Theorem 1.2 of Chapter III in \cite{grillet}) one has $M^{sl}=M/\sim$ where $a\sim b$ provided there exist natural numbers $m,n\geq 1$ and elements $u,v\in M$ such that
$a^m=ub$ and $b^n=va$.
\begin{Co} Let $B$ a submonoid of $A$ and assume for any element $a\in A$ there exist a natural number $n$ such that $a^n\in B$. Then 
$$Spec(A)\to Spec(B)$$
is an isomorphism.
\end{Co}
\begin{proof} It suffice to show that $B^{sl}\to A^{sl}$ is an isomorphism. Take 
any element $a\in A$, since $a\sim a^n$ for all $n$ we see that the map in question is surjective. Now take two elements $b_1,b_2$ in $B$ and assume $b_1\sim b_2$ in $A$. Then there are $u,v\in A$ such that $b_1^k=ub_2$ and $b_2^m=vb_1$. Take $N$ such that $u_1=u^N\in B$ and $v_1=v^N\in B$. Then $b_1^{kN}=u_1b_2^N$ and $b_2^N=v_1b_1^N$. Thus $b_1\sim b_2$ in $B$ and we are done.
\end{proof}
\section{Spectrum of semilattices} 
By the reduction lemma \ref{red} the study of the spectrum of commutative monoids, reduces to the study of the spectrum of semilattices. Before we go further let us fix some terminology.

\subsection{Adjoint maps of morphisms of semilattices} 
The main result of this section is probably well-known to experts, but I couldn't find any references.

For a poset $P$ we let  $P^{op}$ be  the poset, which is $P$ as a set,
but with the reverse ordering.  We sometimes use the notations $$Max\{x\mid x\in P\} \
\  {\rm and } \ \ \ Min\{x\mid x\in P\} $$
for the greatest and least
elements of $P$. 

Let  $X$ and $Y$ be posets. Assume $f:X\to Y$ and $g:Y\to X$ be
maps. We say that $g$ is a  left adjoint of $f$ and $f$ is a
\emph{right adjoint} of $g$, if
$$\left (y\leq f(x) \right)\, \Leftrightarrow \,
\left (g(y)\leq x)\right), \ \ \ x\in X, \, y\in Y.$$

As a specialization of the well-known facts from the 
category theory we can conclude that if $f$ has a left or right adjoint then it is unique. We sometimes write  $g=f^*$ for the left adjoint of $f$ and $f=g^\dagger$ for the right adjoint of $g$. 

\begin{Pro}\label{ra}  Let $X$ and $Y$ be posets. \begin{enumerate}
\item  A map $f:X\to Y$
has a right adjoint $f^*$ if and only if for any $y\in Y$ the set
$\{x\in X | y\leq f(x)\}$ has the least element.
 If this is the case, then
$$f^*(y)=Min\{x\in X | y\leq f(x)\}.$$
If this is the case, then both $f$ and $g$ are monotonic maps.

\item A map $f:X\to Y$ has a left adjoint $f^\dagger:Y\to X$ if and only if for any
$y\in Y$ the set $\{x\in X | f(x)\leq y\}$ has the greatest element. If this is the case, then
$$f^*(y)=Max\{x\in X | f(x)\leq y\}.$$
If this is the case, then both $f$ and $g$ are monotonic maps.
\end{enumerate}
\end{Pro}

A \emph{lattice} is a poset $L$ which is simultaneously a join and meet semi-lattice (that is the opposite poset is a join semi-lattice).

\begin{Le}\label{g_ex} If $L$ is a finite join semi-lattice, then $L$ is a lattice.
\end{Le}

\begin{proof} First we show that $L$ posses the greatest element. In fact, if $L=\{x_1,\cdots,x_n\}$,
then $x_1\vee \cdots \vee x_n$ is the greatest element. Now for any
$a,b\in L$ we consider the subset
$$Q_{a,b}:=\{x\in L| x\leq a, \ {\rm and} \ x\leq b\}.$$
It is clear that the least element belongs to $Q_{a,b}$, so it is
nonempty. It is also clear that $Q_{a,b}$ is a  join
subsemilattice of $L$. Thus $Q_{a,b}$ has the greatest element
$a\wedge b$ and we are done.
\end{proof}

If $f:L\to L'$ is a morphism of  finite join semilattices then as we
have seen $L$ and $L'$ are lattices, but in general $f$ needs not to
be a morphism of lattices. However we have the following
\begin{Pro}  Let $L$ and $L'$ be lattices. \begin{enumerate}

\item  Let $f:L\to L'$ be a morphism of join semi-lattices.
If $L$ is  finite, then $f$ considered as a morphism of posets has a right adjoint
$f^\dagger:L'\to L$ which is a morphism of meet semi-lattices.

\item Let $f:L\to L'$ be a morphism of meet semi-lattices.
If $L'$ is finite, then $f$ considered as a morphism of posets has a left adjoint $f^*:L'\to L$
which is a morphism of join semi-lattices.
\end{enumerate}
\end{Pro}
\begin{proof} Take an element  $y\in L'$ and consider  the set
$$Q_{f,y}=\{x\in L| fx\leq y\}$$
We claim that  $Q_{f,y}$ is a join subsemilattice of $L$. First, the
least element of $L$ belongs to $Q_{f,y}$ and next,  if $x_1,x_2\in
Q_{f,y}$, then $f(x_i)\leq y$ for $i=1,2$. Hence
$$f(x_1\vee x_2)=f(x_1)\vee f(x_2)\leq y$$
and $x_1\vee x_2\in Q_{f,y}$. By Lemma \ref{g_ex} $Q_{f,y}$ has  the
greatest  element.  Hence we can use Proposition \ref{ra} to deduce
existence of $f^\dagger$. Take now $y_1,y_2\in Y$. Then we have
$\left( fx\leq y_1\wedge y_2\right ) \,\Leftrightarrow \, \left(
f(x)\leq y_1 \ {\rm and} \ f(x)\leq y_2\right )\,
 \Leftrightarrow \, \left(x\leq f^\dagger (y_1) \ {\rm and }\
 x\leq f^\dagger(y_2)\right ) \Leftrightarrow \, \left(x\leq f^\dagger (y_1)\wedge
  f^\dagger (y_2)\right )$. Hence
  $$f^\dagger(y_1\wedge y_2)=f^\dagger(y_1)\wedge f^\dagger(y_2)$$
  and the first assertion is proved. A similar argument shows the second assertion.

\end{proof}

\begin{Co} The category of finite join semi-lattices is contravariantly
equivalent to the category of finite meet semi-lattices.
\end{Co}
\begin{proof} By Lemma \ref{g_ex} we can assume that both categories in
question have finite lattices as objects.  The duality functor 
from the category of finite join semi-lattices to the category of
finite meet semi-lattices is the identity on objects and sends a
join semi-lattice morphism $f$ into $f^\dagger$. The functor in the
opposite direction is also identity on objects and sends a meet
semi-lattice morphism $g$ into $g^*$. 
\end{proof}

\subsection{Spec of finite semilattices}

Lemma \ref{red} reduces the study of $Spec(M)$ to the case when $M$ is a join semilattice. Our aim is now to prove that for finite (=finitely generated) semilattice $L$ there exist a ''natural'' homeomorphism $\alpha_L:L\to Spec(L)$, where $L$ is equipped with a topology, where open sets are exactly ideals of $L$.

Let $L$ be a join semilattice considered as a monoid. 
Take an element  $a\in L$ and  consider the set
 $$Q_a=\{x\in L| x\leq a\}.$$
 Let $\alpha(a)$ be the complement of $Q_a$ in $L$.
 Then we have
 
 \begin{Le} \begin{enumerate}
 \item $Q_a$ is a subsemilattice.

 \item $\left (a\leq b\right )\Longleftrightarrow \left ( Q_a\subset Q_b \right)$

 \item $\alpha(a)\in Spec(L)$

 \item $\alpha_L: L\to Spec(L)$ is an injective map.
 
 \item If $L$ is lattice, then $\alpha_L(a\wedge b)=\alpha_L(a)\cup \alpha_L(b)$.
 
 \end{enumerate}

\end{Le}

 \begin{Th}\label{mt}  Let $L$ be  a finite semilattice. \begin{enumerate}
 \item   For any prime ideal $\frak{p}$ we
 let $\beta_L(\frak{p})$ be the greatest element of the subsemilattice $L-\frak{p}$, then for the map $\beta_L:Spec(L)\to L$ one has  $\beta_L\circ \alpha_L=id$. Thus $$\alpha_L:L\to Spec(L)$$
 is bijective. 
 
 \item For  a subset $S\subset L$, the set $\alpha_L(S)$ is an open subset of $Spec(L)$ if and only if $S$ is an ideal.
 
 \item Let  $f:L\to L'$ be  a morphism of finite join semilattices then one has the following commutative
 diagram
 $$\xymatrix{L'\ar[r]^{f^\dagger}\ar[d]_{\alpha_{L'}}& L\ar[d]^{\alpha_L}\\
 Spec(L')\ar[r]_{f^{-1}}& Spec(L)}$$
 where as usual $f^\dagger$ denotes the right adjoint of $f$.
  \end{enumerate}
 \end{Th}

\begin{proof} Take $a\in L$. Then we have
$$\beta_L\circ \alpha_L(a)=\beta_L(L-Q_a)=Max\{x|x\in Q_a\}=a$$
and the first assertion follows. We also have

$$D(a)=\{ \p\in Spec(L)|a\not \in \p\}
=\{\alpha_L(m)|a\not \in \alpha_L(m)\}=$$
$$ \{\alpha_L(m)|a \in Q_m\}=\{\alpha_L(m)|m\in(a)\}$$
Thus $\alpha^{-1}(D(a))=(a)$ and the second assertion also follows.

To see that the diagram commutes it suffice to prove that
$$\beta_L\circ f^{-1}\circ \alpha_{L'}=f^\dagger$$
Take $y\in L'$. Then
$\beta_L\circ f^{-1}\circ \alpha_{L'}(y)=\beta_Lf^{-1}(L'-Q_y)=\beta_L(L-f^{-1}(Q_y))=
Max\{x\in L|f(x)\in Q_y\}=Max\{x\in L| f(x)\leq y\}=f^\dagger(y)$
and we are done.
\end{proof}

\begin{Co}\label{spsp} \begin{enumerate}
\item If $L$ is a finite semi-lattice then the composite
$$L\xto{\alpha_L} Spec(L) \xto{\alpha_{Spec(L)}} Spec(Spec(L))$$
is an isomorphism of semi-lattices, which is functorial in $L$.
\item Any finite semi-lattice is the spectrum of a monoid.
\end{enumerate}
\end{Co}

\begin{proof} The second part is a direct consequence of the first one. Since the bijection $\alpha_L$ reverses the ordering the first part follows.
\end{proof}

Combine Theorem \ref{mt} and Theorem \ref{red} to obtain the following result.
\begin{Co} \label{shedegi} \begin{enumerate}
\item  For any monoid $M$ there is an injective map
$$\alpha_M:M^{sl}\to Spec(M)$$
which is bijective if $M$ is a finitely generated monoid.
\item Let  $M$ be  a finitely generated monoid  and let $f:M\to N$ be  a morphism of  monoids. Then one has the following commutative diagram
$$\xymatrix{N^{sl}\ar[r]^{f^\dagger}\ar[d]_{\alpha_{N}}& M^{sl}\ar[d]^{\alpha_M}\\ 
Spec(N)\ar[r]_{f^{*}}& Spec(M)}$$
where as usual $f^\dagger$ is the right adjoint to the poset map $f_*:M^{sl}\to N^{sl}$.
\item For any finitely generated monoid $M$ one has the natural bijection
$$Spec^3(M)\cong Spec(M),$$
where $Spec^3$ is the three fold composite of $Spec$.
\end{enumerate}

\end{Co}
\subsection{Infinite semilatices and not-finitely generated monoids} Asumume $L$ is a semilatice. We let $L_\lambda$ denote the collection of all finitely generated subsemilattices of $L$. Observe that each $L_\lambda$ is finite and they form a filtered system with respect to inclusion. If $f_{\lambda,\eta}:L_\lambda\subset L_\eta$ is an inclusion of finite subsemilattices, we have induced surjective morphism of dual meet-semilatices
$$f_{\lambda,\eta}^\dagger:L_\eta\to L_\lambda$$
In this way one obtains the inverse system of finite meet-semilattices. Let $L_\infty$ be the inverse limits of this system
\begin{Pro} For any $L$ there is a bijection
$$L_\infty\to Spec(L)$$
\end{Pro}
\begin{proof} Since $L=colim \, L_\lambda$ the result follows from Corollary \ref{zg}.
\end{proof} 

In this way we obtain a  contravariant  functor $L\mapsto L_\infty$ from the category of semilattices to the category of profine semilattices. 
For $L=M^{sl}$ we write  $M_\infty^{sl}$ for $L_\infty$. Thus $M\to M_\infty^{sl}$ is a contavariant functor from monoids to the category of
profinite semilatices. Combining the previous results we obtain a bijection
$$Spec(M)\to M^{sl}_\infty$$
for any commutative monoid $M$.

\end{document}